\documentclass[12pt]{amsart}
\usepackage{amscd,amsmath,amssymb,amsfonts,a4}
\usepackage{pstricks}
\usepackage{color}

\newlength{\rulebreite}

\def\timesover#1#2#3{\ \xymatrix@1@=0pt@M=0pt{ _{#1}&\times&_{#2} \\& ^{#3}&}\ }
\def\otimesover#1#2#3{\ \xymatrix@1@=0pt@M=0pt{ _{#1}&\otimes&_{#2} \\& ^{#3}&}\ }
\usepackage[all]{xy}
\theoremstyle{plain}
\newtheorem{thm}{Theorem}
\newtheorem{lem}[thm]{Lemma}

\theoremstyle{definition}

\newtheorem{rmk}[thm]{Remark}

\numberwithin{equation}{section}

\newcommand{\Pic}{{\rm Pic}}
\newcommand{\Spec}{{\rm Spec \,}}

\newcommand{\inj}{\hookrightarrow}

\newcommand{\sE}{{\mathcal E}}
\newcommand{\sL}{{\mathcal L}}
\newcommand{\sM}{{\mathcal M}}
\newcommand{\sN}{{\mathcal N}}
\newcommand{\sO}{{\mathcal O}}

\newcommand{\G}{{\mathbb G}}
\newcommand{\Z}{{\mathbb Z}}

\begin{document}

\title[Brauer group of desingularized moduli spaces]{The Brauer group
of desingularization of moduli spaces of vector bundles over a curve}

\author[I. Biswas]{Indranil Biswas}

\address{School of Mathematics, Tata Institute of Fundamental
Research, Homi Bhabha Road, Bombay 400005, India}

\email{indranil@math.tifr.res.in}

\author[A. Hogadi]{Amit Hogadi}

\address{School of Mathematics, Tata Institute of Fundamental
Research, Homi Bhabha Road, Bombay 400005, India}

\email{amit@math.tifr.res.in}

\author[Y. I. Holla]{Yogish I. Holla}

\address{School of Mathematics, Tata Institute of Fundamental
Research, Homi Bhabha Road, Bombay 400005, India}

\email{yogi@math.tifr.res.in}



\date{}

\begin{abstract}
Let $C$ be an irreducible smooth projective curve, of genus
at least two, defined over an algebraically closed field of
characteristic zero. For a fixed line bundle ${\mathcal L}$ on $C$,
let $M_C(r,{\mathcal L})$ be the coarse moduli space of 
semistable vector bundles $E$ over $C$ of rank $r$ with $\bigwedge^r 
E\,=\, {\mathcal L}$. We show that the Brauer group of any
desingularization of ${M}_C(r,{\mathcal L})$ is trivial. 
\end{abstract}

\maketitle

\section{Introduction} 

Let $k$ be an algebraically closed field of characteristic zero. Let 
$C/k$ be an irreducible smooth projective curve of genus $g$, with
$g\,\geq\, 2$. Let $\sL\,\in\, 
\Pic^d(C)$ be a line bundle o $C$ of degree $d$; fix an integer
$r\, \geq\, 2$. Let $M\, =\, {M}_C(r,{\mathcal L})$ denote
the moduli space of semistable vector bundles on $C$ of rank $r$ 
and determinant $\sL$. It is known that $M$ is a unirational normal
projective variety. Up to isomorphism $M$ depends only on the class of 
$d$ modulo $r$ and not on the actual line bundle $\sL$. This variety $M$ is
known to be rational if $r$ is coprime
to $d$ \cite[p. 520, Theorem 1.2]{kingsch}; except for the only
case of $g\,=\, r\,=\, d
\,=\, 2$ when $M$ is known to be ${\mathbb P}^3_k$, in all other
cases, where $r$ is not coprime
to $d$, it is unknown whether $M$ is stably rational.

For any projective variety $X/k$ to be rational (or even stably 
rational), it is necessary for the Brauer group ${\rm Br}(\widetilde{X})$ to vanish, where 
$\widetilde{X}\longrightarrow X$ is a desingularization. This motivated us to study
the Brauer group of the desingularization of $M$.

The following result is proved here:

\begin{thm}\label{mainthm}
Let $\widetilde{M}\,\longrightarrow\, {M}_C(r,{\mathcal L})$ be any 
desingularization of ${M}_C(r,{\mathcal L})$. Then ${\rm 
Br}(\widetilde{M})$ is trivial.
\end{thm}

We now give a brief idea of the proof of it. For any possible nonzero 
class $\alpha \,\in\, {\rm Br}(\widetilde{M})\setminus\{0\}$, we show
that there exists a discrete valued field $K$, and a morphism 
$\varphi\, :\, \Spec(K)\,\longrightarrow\, \widetilde{M}$, such
that $\varphi^*\alpha\, \in\, {\rm Br}(K)$ is ramified. 
This morphism $\varphi$
is constructed explicitly out of a suitable 
family of vector bundles on $C$ parameterized by a $\G_m$--gerbe over 
$\Spec(K)$. On the other hand, since $\widetilde{M}$ is a proper 
variety, for any $\xi\, \in\, {\rm Br}(\widetilde{M})$, the
pullback $\varphi^*\xi\, \in\, {\rm Br}(K)$ must be unramified 
(this is because the morphism $\Spec(K)\,\longrightarrow\, M$ extends
to the discrete valuation ring). Thus $\alpha$ must be zero.

Theorem \ref{mainthm} was proved earlier in \cite{Ni} under the
assumption that $r\, =\, 2$ (see \cite[p. 309, Theorem 1]{Ni});
this theorem of Nitsure was also proved later in \cite{Ba}.
When $r\, =\, 2$, explicit desingularizations of $M$ are
available; these desingularizations are crucially used in \cite{Ni},
\cite{Ba}.

\medskip
\noindent {\bf Acknowledgement.} We thank the referee for pointing out an 
incorrect statement in an earlier version and for helpful comments. We thank 
Najmuddin Fakhruddin 
and the referee for independently pointing out that an argument in an 
earlier version involving resolution of singularities of stacks could be 
avoided altogether.

\section{The stable locus}\label{se2}

We continue with the notation of the introduction.

Let $\sM\, =\, {\sM}_C(r,{\mathcal L})$ be the moduli stack of 
semistable vector bundles over $C$ of rank $r$ and determinant
$\mathcal L$. Let
\begin{equation}\label{p}
p\, :\, {\sM}\, \longrightarrow\, M
\end{equation}
be the natural morphism to the earlier defined moduli space $M$;
this $M$ is a good moduli space 
for $\sM$ in the sense of \cite[p. 10, Definition 4.1]{jarod}.

The following notation and comments will be used.
\begin{enumerate}
\item The degree of $\sL$ will be denoted by $d$, and
${\rm gcd}(r,d)$ will be denoted by $n$.

\item Let $M_0\,\subset\, M$ 
(respectively, $\sM_0\subset \sM$) be the Zariski open subset 
(respectively, sub-stack) parameterizing stable bundles. We note
that $M_0$ is contained in the smooth locus of $M$. In fact
$M_0$ coincides with the smooth locus of $M$ except for the only
case of $g\,=\, r\,=\, d
\,=\, 2$ (in this case $M$ is known to be smooth).

\item There is a natural inclusion of $\G_m$ as a normal subgroup in the isotropy group of
any point of $\mathcal M$ corresponding to action of $\G_m$ on vector bundles 
given by scalar multiplications. 
For stable vector bundles, the isotropy group 
coincides $\G_m$. Let
$$
F\,:\,\sM \,\longrightarrow\, \sN
$$
be a $1$-morphism obtained by rigidifying
$\G_m$ (see \cite[pp. 3572--3573, Theorem 5.1.5]{daa}), meaning
$F$ a $1$-morphism defining a $\G_m$--gerbe
such that for every point $z$ 
of $\sM$, the kernel of the homomorphism induced between the
isotropy groups at $z$ and $F(z)$ is precisely $\G_m$. $\sN$ is a smooth stack which is generically a scheme.

\item Choose a stable vector bundle $W$ of rank $r/n$ and degree $d/n$
on $C$. Let $z_0'$ be 
the $k$--point of $\sN$ which corresponds to the 
vector bundle
$$
E_0\,=\, W^{\oplus n}\, .
$$
We let $z_0$ be the image of $z_0'$ in $M$; this $z_0$ is also a 
$k$--point.

\item Let $\pi\,:\,\widetilde{M}\,\longrightarrow\, M$ be a 
desingularization which is an 
isomorphism outside the singular locus of $M$; in particular, it is an 
isomorphism over $M_0$. 
We thus have the following 
diagram.
$$\xymatrix{
   & \sM\ar[d] \\
  & \sN \ar[d]^\theta \\
  \widetilde{M}\ar[r]_{\pi} & M
}
$$
\end{enumerate}

\begin{rmk}
Since the Brauer group is a birational invariant for smooth projective 
varieties, proving Theorem \ref{mainthm} for one particular
desingularization of $M$ is equivalent to proving it for all
desingularizations of $M$. Thus it is enough to prove Theorem
\ref{mainthm} for the desingularization $\widetilde{M}$ chosen in (5). 
\end{rmk}

\begin{lem}\label{curve}
Given any $K$--point $x_0\, \in \,\sN(K)$, where $K/k$ is
a field extension,
there exists a smooth curve $Y/K$, a $K$--point $y_0\,\in\, Y(K)$, and
a map $$\psi\,:\,Y\,\longrightarrow\, \sN\, ,$$ such that $\psi(y_0)\,
=\, x_0$ and $\psi(Y)\bigcap M_0\,\neq\, \emptyset$.
\end{lem}

\begin{proof}
By \cite[p. 49, Th\'eor\`eme 6.3]{laumon}, we can choose an
atlas $\pi\,:\,U\, \longrightarrow
\, \sN$ such that $x_0$ lifts to a $K$--point $\widetilde{x}_0$ of
$U$. Since $\sN/k$ is smooth, so is $U/k$. Thus, a
general complete intersection curve in $U/k$ passing
through $\widetilde{x}_0$ satisfies the conditions.
\end{proof}

\begin{lem}\label{chooseK}
Fix any integer $n\, \geq\, 2$.
There exists a field extension $K/k$, a $k$--discrete valuation $v$ on 
$K$, and a central division algebra $D/K$ of index $n$, such that for 
any integer $\ell$, the class $\ell\cdot[D]\, \in\, {\rm Br}(K)$ is 
unramified at $v$ if and only if $\ell$ is divisible by $n$. (Here
$[D]$ denotes the class in ${\rm Br}(K)$ defined by $D$.)
\end{lem}

\begin{proof}
Set $K\,=\, k(x,y)$ to be the purely transcendental extension. Let $v$ 
be the valuation given by the height one prime ideal 
$$ (x) \,\subset\, k[x,y]\, $$
and let $L=k(y)$ denote the residue field of $K$ at $v$.
Set $D$ to be the cyclic algebra $(x,y)_{\zeta}$, 
where $\zeta$ is any chosen primitive $n$--th root of unity. The
obstruction for an $n$--torsion class in ${\rm Br}(K)$ to
be unramified is measured by the tame symbol 
$$
H^2(K,\mu_n)\, \longrightarrow\, H^1(L,\, \Z/n)\, \cong \,L^*/L^{*n}\, .
$$
Here we identify $H^2(K,\,\mu_n)$ with the $n$--torsion subgroup of
${\rm Br}(K)$. Note that the isomorphism $H^1(L,\Z/n)\cong L^*/L^{*n}$ depends on the choice of $\zeta$. 
The image of $[D]$ in $L^*/L^{*n}$ is the class
defined by $y^{-1}$ (see \cite[Example 7.1.5 and Corollary 7.5.3]{gille}) which has
order $n$. Hence $\ell\cdot[D]$ is unramified at $v$
if and only if $\ell$ is divisible by $n$. Moreover, since the order of $[D]$ is equal to its index, $D$ is a division algebra.
\end{proof}

\begin{lem}\label{brbound}
There is a natural inclusion ${\rm Br}(\widetilde{M})\,\inj\, 
{\rm Br}(M_0)$, where $M_0$ is defined in (2).
\end{lem}

\begin{proof}
Since $\widetilde{M}$ is smooth, the homomorphism ${\rm Br}
(\widetilde{M})\,\longrightarrow\, {\rm Br}(\pi^{-1}(M_0))$ induced
by the open embedding $\pi^{-1}(M_0)\,\hookrightarrow\, \widetilde{M}$ is
injective. Now the lemma follows immediately from the assumption that 
the morphism
$$
\pi\vert_{\pi^{-1}(M_0)}\,:\,\pi^{-1}(M_0)\,\longrightarrow\, M_0
$$
is an isomorphism.
\end{proof}

\section{Proof of Theorem \ref{mainthm}}

If $g\, =\, 2\, =\, r$, and $d$ is even, then $M\, =\, {\mathbb P}^3_k$
\cite[pp. 33--34, Theorem 2]{NR}; hence
Theorem \ref{mainthm} holds in this case.
If $g\, =\, 2\, =\, r$, and $d$ is odd, then $M$ is a smooth projective 
rational variety
as $n\,=\, 1$ (see (1) of Section \ref{se2});
so Theorem \ref{mainthm} holds in this case also. Hence we assume that
$g\, \geq\, 3$ if $r\, =\, 2$.

Consider the $\G_m$--gerbe $\sM\,\longrightarrow\, \sN$ in (3)
of Section \ref{se2}. Let
$$
\alpha\,\in\, {\rm Br}(\sN)
$$
be the class defined by it. Since 
$\theta\, :\, \sN\,\longrightarrow\, M$
is an isomorphism over $M_0$, we
consider $M_0$ also as an open subset of $\sN$. Thus
\begin{equation}\label{ap}
\alpha'\,:=\,\alpha\vert_{M_0}
\end{equation}
defines an element of ${\rm Br}(M_0)$. This class $\alpha'\,\in\,
{\rm Br}(M_0)$ generates ${\rm Br}(M_0)$, and its
order is precisely $n$ \cite[p. 267, Theorem 1.8]{bbgn}.

Since $M_0$ can also be identified with an open subset of 
$\widetilde{M}$ (see (5) of Section \ref{se2}), it makes 
sense to ask whether a class in ${\rm Br}(M_0)$ extends to a class in 
${\rm Br}(\widetilde{M})$. In view of Lemma \ref{brbound}, and the above
description of ${\rm Br}(M_0)$, in order to prove Theorem 
\ref{mainthm} it suffices to show the following:

\medskip
\textbf{Statement A.}\, \textit{For a given integer 
$\ell$, the class $\ell\alpha'$ extends to an element of ${\rm 
Br}(\widetilde{M})$ only if $\ell$ is a multiple of $n$ (or equivalently if $\ell\alpha'$ vanishes).} 
\medskip

We will prove Statement A in three steps.

\medskip
\noindent
\underline{Step $1$}:\, Let $K/k$ be a field extension, and let $D/K$ 
be a central division algebra of index $n$, given by Lemma 
\ref{chooseK}. Let $X\,\longrightarrow\, \Spec(K)$ be the
$\G_m$--gerbe defined by the 
class $[D]\,\in\, H^2(K,\G_m)$. Then there exists a twisted bundle $V$ of rank $n$ 
on $X$ such that $\sE nd_{\sO_X}(V)$ descends to the coherent sheaf
on $\Spec(K)$ corresponding to $D$. Consider the vector bundle 
$$F\,:= \, V\boxtimes W $$
on $X\times_kC$, where $W$ is the vector bundle in (4)
of Section \ref{se2}. This $F$ 
can be thought of as a family of semistable 
vector bundles on $C$ parameterized by $X$. Hence we get a morphism 
$$
f\, :\, X\,\longrightarrow\, \sM
$$
representing this family. This 
morphism $f$ induces a morphism 
\begin{equation}\label{phi}
x_0\,:\, \Spec(K) \,\longrightarrow\, \sN\, .
\end{equation}
By the construction of $x_0$, it has the following properties:
\begin{enumerate}
 \item $x_0^*\alpha\, =\, [D]$, and
 \item the image of $x_0$ is equivalent to the point $z_0'$ that 
corresponds to $E_0$ in (4) of Section \ref{se2}.
\end{enumerate}

\medskip

\noindent \underline{Step $2$}:\, Using Lemma \ref{curve} we
find a smooth curve $Y/K$, a $1$-morphism $\psi\,:\,Y\,\longrightarrow
\, \sN$ and $y_0\,\in\, Y(K)$, such that $$\psi(y_0)\,=\, x_0\, .$$
Consider the map from $Y\to M$ given by the composition $Y\to \sN 
\stackrel{\theta}{\to} M$. By 
Lemma \ref{curve},
the generic point of $Y$ maps into $M_0\subset M$ and hence lifts to $\widetilde{M}$. Since 
$$\pi:\widetilde{M}\to M$$ is proper, by using the valuative criterion of 
properness we see that the map $Y\to M$ 
factors through $\widetilde{M}$.  \\

We thus have a commutative diagram 
$$\xymatrix{
Y \ar[r]^{\psi}\ar[d]_h & \sN \ar[d]^{\theta} \\
\widetilde{M}\ar[r]^{\pi} & M 
}$$

\medskip
\noindent \underline{Step $3$}:\, Now suppose for some integer $\ell$, 
the class $\ell\cdot\alpha'\, \in\, {\rm Br}(M_0)$ (see \eqref{ap}) 
extends to a class $\beta$ on entire $\widetilde{M}$. 
The two classes $\ell\cdot\alpha\,\in\, {\rm Br}(\sN)$ and 
$\beta\,\in\, {\rm Br}(\widetilde{M})$ coincide when restricted to
$M_0$, which is canonically identified with an open subset of
$\sN$ as well as of $\widetilde{M}$.

We claim that 
\begin{equation}\label{eq}
\psi^* (\ell\cdot\alpha)\,=\, h^*\beta
\, \in\, {\rm Br}(Y)\, .
\end{equation}

The claim follows since $Y$ is a regular integral scheme and the above 
Brauer classes coincide on the dense open subset $h^{-1}(M_0)$ of $Y$.

Thus, restricting the above classes in $Br(Y)$ to the $K$-point 
$y_0:\Spec(K)\to Y$, and using $\psi\circ y_0=x_0$, we get that
$$x_0^*(\ell\cdot \alpha) \,=\, y_0^*(h^*\beta)
\,\in\, {\rm Br}(K)\, .$$

However, $\widetilde{M}$ being a proper variety, the morphism
$$  y_0\circ h \,:\,\Spec(K)\,\longrightarrow
\, \widetilde{M}$$
extends to a morphism
$$
\Spec(R)\,\longrightarrow\, \widetilde{M}\, ,
$$
where $R$ is the discrete valuation ring corresponding to the valuation
$v$ on $K$. Thus the class $y_0^*(h^*\beta) \,\in\, {\rm
Br}(K)$ is unramified at $v$.
But by construction of $x_0$ in Step 1, the class
$x_0^*(\ell\cdot\alpha)$ coincides with $\ell\cdot[D]$, which, by
Lemma \ref{chooseK}, is ramified at $v$ unless $\ell$ is
divisible by $n$. This implies that $\ell$ is divisible by $n$.
This proves Statement A, and hence the proof
of Theorem \ref{mainthm} is complete.

\begin{rmk} Since the moduli space $M$ is locally factorial, it
follows that the natural homomorphism ${\rm Br}(M)\,\longrightarrow
\, {\rm Br}(M_0)\,$ is injective. From Statement A we know
that no nonzero class $\ell \alpha'\, \in\, {\rm Br}(M_0)$
extends to ${\rm Br}(\widetilde {M})$, hence such a class
cannot extend to ${\rm Br}(M)$. Consequently,
the Brauer group of the moduli space $M$ vanishes.
\end{rmk}

\end{document}